\documentclass[11pt]{article}
\usepackage[letterpaper, total={6.5in, 9in}]{geometry}
\usepackage[english]{babel}
\usepackage[utf8]{inputenc} 
\usepackage[T1]{fontenc}    
\usepackage[colorlinks=true,linkcolor=blue]{hyperref}       
\usepackage{url}            
\usepackage{booktabs}       
\usepackage{amsfonts}       
\usepackage{nicefrac}       
\usepackage{microtype}      
\usepackage{geometry}
\usepackage{subcaption}
\usepackage{graphicx}

\usepackage[ruled,linesnumbered]{algorithm2e}
\usepackage{amssymb}
\usepackage{amsfonts}
\usepackage{amsmath}
\numberwithin{equation}{section}
\numberwithin{equation}{section}
\usepackage{amsthm}
\usepackage{latexsym}
\usepackage{epsfig}
\newtheorem{theorem}{Theorem}[section]
\newtheorem{corollary}[theorem]{Corollary}
\newtheorem{lemma}[theorem]{Lemma}

\newtheorem{proposition}[theorem]{Proposition}

\newtheorem{definition}{Definition}[section]

\newtheorem{remark}[theorem]{Remark}
\newtheorem{assumption}{Assumption}

\usepackage{enumitem}

\usepackage{adjustbox}
\usepackage{pdfcomment}
\usepackage{todonotes}
\usepackage{setspace}
\usepackage{array}
\usepackage{wrapfig}

\newcommand{\linV}{\vec{V}}
\newcommand{\linVp}{\vec{V}_p}
\newcommand{\linVd}{\vec{V}_d}
\newcommand{\lin}[1]{\vec{#1}}

\newcommand{\dist}{\operatorname{Dist}}
\newcommand{\gap}{\operatorname{Gap}}

\newcommand{\thetax}{\theta_p^\star}
\newcommand{\thetas}{\theta_d^\star}
\newcommand{\oneratio}{Limiting Error Ratio}
\newcommand{\eratio}{error ratio}
\newcommand{\LPsharp}{LP sharpness}

\newcommand{\ignore}[1]{}

\usepackage{xcolor}

\newcommand{\condL}{\mathcal{L}}

\newcommand{\eps}{\varepsilon}

\newcommand{\calE}{{\cal E}}
\newcommand{\calF}{{\cal F}}

\newcommand{\calL}{{\cal L}}

\newcommand{\calS}{{\cal S}}

\newcommand{\calX}{{\cal X}}

\title{On the Relation Between LP Sharpness and Limiting Error Ratio and Complexity Implications for Restarted PDHG}
\author{Zikai Xiong\thanks{MIT Operations Research Center, 77 Massachusetts Avenue, Cambridge, MA 02139, USA. 
\href{mailto:zikai@mit.edu}{zikai@mit.edu}.  Research supported by AFOSR Grant No. FA9550-22-1-0356.
}  
\and Robert M. Freund\thanks{MIT Sloan School of Management, 77 Massachusetts Avenue, Cambridge, MA 02139, USA. 
\href{mailto:rfreund@mit.edu}{rfreund@mit.edu}. Research supported by AFOSR Grant No. FA9550-22-1-0356.}}
\date{June 2025}

\begin{document}

\maketitle

\begin{abstract}
	
	There has been a recent surge in development of first-order methods (FOMs) for solving huge-scale linear programming (LP) problems.  The attractiveness of FOMs for LP stems in part from the fact that they avoid costly matrix factorization computation. However, the efficiency of FOMs is significantly influenced -- both in theory and in practice -- by certain instance-specific LP condition measures. Xiong and Freund recently showed that the performance of the restarted primal-dual hybrid gradient method (PDHG) is predominantly determined by two specific condition measures: \LPsharp~and \oneratio.  In this paper we examine the relationship between these two measures, particularly in the case when the optimal solution is unique (which is generic -- at least in theory), and we present an upper bound on the \oneratio~involving the reciprocal of the \LPsharp. This shows that in LP instances where there is a dual nondegenerate optimal solution, the computational complexity of restarted PDHG can be characterized solely in terms of \LPsharp~and the distance to optimal solutions, and simplifies the theoretical complexity upper bound of restarted PDHG for these instances.
\end{abstract}

\section{Introduction}\label{secpreliminaries}

In this work, we study linear programming (LP) problems in the following form:
\begin{equation}\label{pro primal LP}
	\begin{aligned}
	 	\underset{x\in\mathbb{R}^n}{\operatorname{min}} \ & \ c^\top x \\
 	\operatorname{s.t.} \ & \ x \in V_p, \ x \ge 0 
	\end{aligned}
\end{equation}
where $V_p$ is an affine subspace in $\mathbb{R}^n$. For common standard-form problems we have  $V_p = \{x \in \mathbb{R}^n: Ax = b\}$ for a given matrix $A\in\mathbb{R}^{m\times n}$ and a right-hand side vector $b\in\mathbb{R}^m$~\cite{bertsimas1997introduction}. Without loss of generality we assume that $c \in \lin{V}_p$ where $\linV_p$ is the linear subspace associated with $V_p$. This is without loss of generality since we can replace $c$ with its projection onto $\linV_p$ and maintain the same optimal solution set and optimal objective function value. We therefore include this condition in our set of assumptions about \eqref{pro primal LP} as follows:
\begin{assumption}\label{assump:general}
	The LP problem \eqref{pro primal LP} has an optimal solution, the objective vector $c$ is in the linear subspace $\lin{V}$, and non-optimal feasible solutions exist for \eqref{pro primal LP} and its dual problem. 
\end{assumption}

Under Assumption~\ref{assump:general}, the primal and dual problems can be written in the following symmetric format:
\begin{equation}\label{pro: primal dual reformulated LP}\tag{PD}
	\begin{aligned}
		&\text{(P)}& \quad\min_{x\in\mathbb{R}^n}& \  c^\top x & \quad\quad&\text{(D)}&\quad\max_{s\in\mathbb{R}^n}& \  - q^\top s \\
		&&\text{s.t.} & \ x \in \calF_p:= V_p\cap\mathbb{R}^n_+
		&\quad & &\text{s.t.} & \ s \in \calF_d:= V_d\cap\mathbb{R}^n_+ \\
		&&&\ \quad  \quad V_p:= q + \lin{V}_p& & & & \quad \quad \ V_d:= c + \lin{V}_d
	\end{aligned}
\end{equation}
where $\linVd$ is the linear subspace associated with the affine subspace $V_d$, and $\linVp$ and $\linVd$ are orthogonal complements, namely,
\begin{equation}
	\mathbb{R}^n = \linVp + \linV_d \quad \text{ and }\quad \linVp \ \bot \ \linV_d \ ,
\end{equation}
and
\begin{equation}\label{eq:norm_of_c_and_q}
	c \in \linVp \text{ and }c = \arg\min_{c_0 \in V_d} \|c_0\|;\quad \text{ also } \quad q \in \linVd \text{  and } q = \arg\min_{q_0\in V_p}\|q\| \ .
\end{equation}
This symmetric reformulation was first proposed in \cite{todd1990centered}. We will denote the primal and dual optimal solutions of \eqref{pro: primal dual reformulated LP}  as $\calX^\star$ and $\calS^\star$, respectively. The duality gap is
$$
\gap(x,s) :=  c^\top x + q^\top s \  ,
$$
and the primal-dual optimal solution set is
$$
\calX^\star\times \calS^\star :=
\left\{
(x,s)
\left|
x\in V_p, \ x \in \mathbb{R}^n_+, \  s\in V_d, \ s \in \mathbb{R}^n_+, \ \gap(x,s) \le 0
\right.
\right\} \ .
$$

Based on the symmetric primal-dual formulation \eqref{pro: primal dual reformulated LP}, \cite{xiong2023geometry} introduce two condition measures: \LPsharp~and \oneratio, and provide new computational guarantees for the restarted Primal-Dual Hybrid Gradient method (PDHG) involving these condition measures. Here we address the question of the relationship between \LPsharp~and \oneratio and whether the overall complexity proven in \cite{xiong2023geometry} can be simplified by replacing one these condition measures with an upper bound involving the other condition measure.  We will demonstrate the following results for the case when the optimal solution of \eqref{pro primal LP} is unique:
\begin{enumerate}
	\item The primal \oneratio~is upper-bounded by the reciprocal of the primal \LPsharp~multiplied by the relative distance to the dual optimal solutions.  Similarly, the dual \oneratio~is upper-bounded by the reciprocal of the dual \LPsharp~multiplied by the relative distance to the primal optimal solutions. However, we show that the reciprocal of the \LPsharp~is not upper-bounded by the \oneratio; this is shown by constructing a specific class of LP instances.
	\item The overall worst-case complexity of restarted PDHG can be expressed using only the \LPsharp~and the relative distance to optimal solutions. Similar to the results in \cite{xiong2023geometry}, the choice of the primal and dual step-sizes in PDHG can affect occurrence of certain cross-terms in the complexity bounds.
\end{enumerate}

We note that various condition numbers (measures) have been introduced for algorithms for LP with the aim of studying/explaining the theoretical and/or practical performance of algorithms for LP, including $\kappa(A)$ (the ratio of largest to smallest singular value of $A$ \cite{boyd2004convex}), the Hoffman constant \cite{hoffman2003approximate}, $\chi(A)$ and $\bar\chi(A)$ \cite{todd1990dantzig, stewart1989scaled}, and Renegar’s distance to ill-posedness \cite{renegar1994some}.  There are also a variety of geometry-focused  condition numbers involving characteristics of level-sets \cite{freund2003primal}, symmetry measures of convex bodies \cite{belloni2008symmetry, epelman2002new}, etc. For the use of these and other condition measures in the context of LP see \cite{vavasis1996primal, todd1990dantzig, epelman2002new} among others.


\paragraph{Notation.}
Throughout this paper, unless explicitly stated otherwise, $\|\cdot\|$ denotes the Euclidean norm, and we use $\| \cdot\|_1$ to denote the $\ell_1$ norm. For any $x \in \mathbb{R}^n$ and $\calX\subset \mathbb{R}^n$, we denote the Euclidean distance from $x$ to the set $\calX$ as $\dist(x,\calX):= \min_{\hat{x} \in \calX} \|x-\hat{x}\|$. The diameter of $\calX$ is denoted as $\operatorname{Diam}(\calX):=\max_{x,\hat{x}\in\calX}\|x - \hat{x}\|$.
For any matrix $A$,  $\sigma_{\max} ^+(A)$ and $\sigma_{\min}^+(A)$ denote the largest and smallest non-zero singular values of $A$, respectively. For an affine subspace $V$, let $\linV$ denote the associated linear subspace of $V$, whereby $V=\linV + v$ holds for every $v \in V$.
We use $\mathbb{R}^n_+$ and $\mathbb{R}^n_{++}$ to denote the nonnegative orthant and strictly positive orthant in $\mathbb{R}^n$, respectively.
For a linear subspace $\linV \subset \mathbb{R}^n$, $\linV^\bot$ denotes the orthogonal complement of $\linV$.  We use $e$ to denote the all-ones vector in $\mathbb{R}^n$, namely $e=(1, \ldots, 1)^\top$.  We use $\| \cdot\|_1$ to denote the $\ell_1$ norm.

\paragraph{Organization.}
This rest of this paper is structured as follows.
In Section \ref{sectwoconditionmeasures} we recall the definitions of the two condition measures introduced in \cite{xiong2023geometry}. In Section \ref{secrelation} we present our main result about the relationship between these two condition measures. 
In Section \ref{secPDHGcomplexity} we present an simplified version of the worst-case complexity analysis for restarted PDHG in \cite{xiong2023geometry} using the results from Section \ref{secrelation}.

\section{Two Geometry-based Condition Measures for LP}\label{sectwoconditionmeasures}

To study the computational guarantees for the restarted Primal-Dual Hybrid Gradient method (PDHG) , \cite{xiong2023geometry} has introduced two condition measures for LP problems \eqref{pro primal LP} that together play an important role in the performance of restarted PDHG both in theory and in practice. The two condition measures under consideration are the \oneratio~and \LPsharp. The first condition measure, \oneratio, is defined as follows, using the primal problem \eqref{pro primal LP} as an illustrative case:

\begin{definition}[\oneratio, Definition 1.1 of \cite{xiong2023geometry}]
	For any $x\in V_p\setminus \calF_p$, the \eratio~of $\calF_p$ at $x$ is defined as:
	\begin{equation}\label{eq:error_ratio}
		\theta(x):= \frac{\dist(x,\calF_p)}{\dist(x,\mathbb{R}^n_+)} \ , 
	\end{equation}
	and for any $x\in\calF_p$ we define $\theta(x):=1$. The \oneratio~is then defined as:
	\begin{equation}\label{eq:def_oner}
		\thetax :=  \lim_{\epsilon \to 0} \left(\sup_{x\in V_p, \dist(x,\calX^\star)\le \epsilon }\theta(x) \right).
	\end{equation}
\end{definition}

And of course we can similarly define \oneratio~for the dual problem in \eqref{pro: primal dual reformulated LP}. We will use $\thetax$ and $\thetas$ to denote the \oneratio~for the primal and dual problems in \eqref{pro: primal dual reformulated LP}.

The second condition measure, \LPsharp, is defined as follows, where again we use the primal problem \eqref{pro primal LP} as the illustrative case:
\begin{definition}[\LPsharp, Definition 1.1 of \cite{xiong2023geometry}]
	Let $f^\star$ denote the optimal objective function value of the problem \eqref{pro primal LP}, and define the hyperplane of the optimal objective value to be $H^\star:= \{x\in\mathbb{R}^n: c^\top x = f^\star\}$. The \LPsharp~is defined as follows: 
	\begin{equation}\label{eq_def_LPsharpness}
		\mu:= \inf_{x\in\calF_p\setminus \calX^\star} \frac{\dist(x,V_p\cap H^\star)}{\dist(x,\calX^\star)} \ .
	\end{equation} 
\end{definition}

Again we can similarly define \LPsharp~for the dual problem in \eqref{pro: primal dual reformulated LP}. For simplicity of notation we will use $\mu_p$ and $\mu_d$ to denote the \LPsharp~of the primal and dual problems in \eqref{pro: primal dual reformulated LP}. Under Assumption \ref{assump:general} the numerator of \eqref{eq_def_LPsharpness} has a closed form:
$$
\dist(x,V_p\cap H^\star) = \frac{c^\top x - f^\star}{\|c\|} = \frac{c^\top x - f^\star}{\dist(0,V_d)} \ ,
$$
where the second equality uses \eqref{eq:norm_of_c_and_q}. With this closed form, the following remark presents an alternative expression for the \LPsharp.

\begin{remark}
	Suppose that Assumption \ref{assump:general} holds. Then
	\begin{equation}\label{eq:equivalent_mu_p}
		\mu_p := \inf_{x \in \calF_p \setminus \calX^\star} \frac{c^\top x - f^\star}{\|c\| \cdot \dist(x,\calX^\star)} = \frac{1}{\dist(0,V_d)} \cdot \inf_{x \in \calF_p \setminus \calX^\star} \frac{c^\top x - f^\star}{ \dist(x,\calX^\star)} \ ,
	\end{equation}
	and
	\begin{equation}\label{eq:equivalent_mu_d}
		\mu_d := \inf_{s \in \calF_d \setminus \calS^\star} \frac{q^\top s + f^\star}{\|q\| \cdot \dist(s,\calS^\star)} = \frac{1}{\dist(0,V_p)} \cdot \inf_{s \in \calF_d \setminus \calS^\star} \frac{q^\top s + f^\star}{ \dist(s,\calS^\star)} \ .
	\end{equation}
\end{remark}
\cite{xiong2023geometry} uses the two measures \oneratio~and \LPsharp~to establish an overall complexity bound for restarted PDHG for LP problems \eqref{pro primal LP}. In the next section we will demonstrate that \oneratio~is upper-bounded by the reciprocal of \LPsharp~under the assumption that \eqref{pro primal LP} has a unique optimal solution.

\section{Relation Between \oneratio~and LP Sharpness}\label{secrelation}

In this section we present our main result that establishes an upper bound on \oneratio~using the reciprocal of \LPsharp~when the LP problem has a unique optimal solution.

\begin{theorem}\label{thm:main}
	For the primal  LP problem \eqref{pro primal LP}, suppose that Assumption \ref{assump:general} holds  and that the optimal solution is unique.  Then the following inequality holds:
	\begin{equation}\label{eq bound of thetax}
		\thetax  \le \frac{1}{\mu_p} \cdot \sqrt{n}  \cdot \left(
		\frac{\dist(0,\calS^\star)}{\dist(0,V_d)}
		+
		 \frac{\operatorname{Diam}(\calS^\star)}{\dist(0,V_d)}
		 \right) \ .
	\end{equation}
\end{theorem}
Examining the two right-most fractions in the right-hand side of \eqref{eq bound of thetax} we can interpret  
$\frac{\dist(0,\calS^\star)}{\dist(0,V_d)}$ as the relative distance (from $0$) to the dual optima set, and we can interpret 
$\frac{\operatorname{Diam}(\calS^\star)}{\dist(0,V_d)}$ as the relative diameter of the dual optimal set, because both of these quantities are (re-)scaled by the distance to the corresponding affine subspace $V_d$. Also observe that the sum of the distance to optima and the diameter is both lower-bounded and upper-bounded by the maximum-norm dual optimal solution, namely 
$$
\max_{s^\star \in \calS} \|s^\star\| \le \dist(0,\calS^\star) +  \operatorname{Diam}(\calS^\star) \ , $$
and also $$
\max_{s^\star \in \calS} \|s^\star\|  \ge \dist(0,\calS^\star) 
$$
and 
$$
 \operatorname{Diam}(\calS^\star) = \max_{u,v\in\calS^\star}\|u-v\| \le 2\cdot \max_{s^\star \in \calS} \|s^\star\|  \ ,
$$ which together imply that 
$$
\max_{s^\star \in \calS} \|s^\star\| \le \dist(0,\calS^\star) +  \operatorname{Diam}(\calS^\star) \le 3 \max_{s^\star \in \calS} \|s^\star\| \ . $$
Overall, the above theorem guarantees that when the primal problem has a unique optimal solution, the value of the primal \oneratio~$\thetax$ cannot be large, provided that the reciprocal of the primal sharpness $\mu_p$ and the (relative) maximum norm of dual optimal solutions are both small.  And of course a  corresponding result also holds for the dual problem.

\begin{corollary}\label{cor bound thetas}
	For the dual problem  in \eqref{pro: primal dual reformulated LP}, suppose that Assumption \ref{assump:general} holds and that the optimal dual solution is unique.  Then the following inequality holds:
	\begin{equation}\label{eq bound of thetas}
		\thetas  \le \frac{1}{\mu_d}\cdot  \sqrt{n}  \cdot \left(
		\frac{\dist(0,\calX^\star)}{\dist(0,V_p)}
		+
		\frac{\operatorname{Diam}(\calX^\star)}{\dist(0,V_p)}
		\right) \ .
	\end{equation}
\end{corollary}

Towards the proof of Theorem \ref{thm:main} we first review two useful results from \cite{xiong2023geometry} The first result is Proposition 2.6 in \cite{xiong2023geometry} which we re-state as follows:

\begin{proposition}[Proposition 2.6 in \cite{xiong2023geometry}]\label{lm upperbound}
	Let $x^\star \in \calX^\star$ and suppose that $\calX^\star \subset\{x:\|x - x^\star\| \le R\}$ for some $R >0$. Then it holds that $\thetax \le B^\star$, where $B^\star$ is defined as follows:
	\begin{equation}
		\begin{aligned}
			B^\star := \inf_{r > 0 , \bar{x} \in \mathbb{R}^n} & \ \frac{R + \| \bar{x} - x^\star\| }{r} \\
			\text{s.t.} \quad & \ \bar{x} \in V_p, \ \bar{x}\ge r \cdot e \ . 
		\end{aligned}
	\end{equation}
\end{proposition}
This proposition states that the \oneratio~cannot be excessively large when the LP problem has a strictly feasible solution that is neither too large nor too close to the boundary of $\mathbb{R}^n_+$. The second result is Theorem 2.1 of \cite{freund2003primal}, which presents a geometric relationship between the primal objective function level sets and the dual objective function level sets. To introduce this result we define two quantities:
\begin{equation}\label{def r delta}
	\begin{aligned}
		r_\delta := & \max_{x}  \quad \min_{i}\{x_i\}  \\ 
		& \ \ \operatorname{s.t.} \quad x \in \calF_p \\
		&  \quad \quad \quad \ \ c^\top x \le f^\star + \delta  \\
	\end{aligned}
\end{equation}
for any $\delta \ge 0$, and 
\begin{equation}\label{def R eps}
	\begin{aligned}
		R_\eps := & \max_{s}  \quad \| s \|_1 \\ 
		& \ \ \operatorname{s.t.} \quad s \in \calF_d \\
		&  \quad \quad \quad \ \ -q^\top s \ge f^\star - \eps  \\
	\end{aligned}
\end{equation}
for any $\eps \ge 0$. The quantity $r_\delta$ is the positivity of the most positive $x$ in the primal objective function $\delta$-optimal level set; or equivalently as the distance to the boundary of the nonnegative orthant of point in the $\delta$-optimal level set that is farthest from the boundary. The quantity $R_\eps$ can be interpreted as the norm of the maximum-norm point $s$ in the dual $\eps$-optimal level set. Using $r_\delta$ in \eqref{def r delta} and $R_\eps$ in \eqref{def R eps}, we re-state Theorem 2.1 of \cite{freund2003primal} as follows:
\begin{lemma}[Theorem 2.1 in \cite{freund2003primal}]\label{lm Rr lemma}
	Suppose that the optimal objective value $f^\star$ of \eqref{pro primal LP} is finite. If $R_\eps$ is positive and finite, then 
	\begin{equation}\label{eq Rr}
		\min\{\eps,\delta\} \le R_\eps\cdot r_\delta \le \eps + \delta \		 .
	\end{equation}
	Otherwise, $R_\eps = 0$ if and only if $r_\delta = + \infty$, and $R_\eps = +\infty$ if and only if $r_\delta = 0$.  
\end{lemma}
This result states, for instance, that the quantities $r_\delta$ and $R_\eps$ are inversely related to within a factor of $2$ in the case when $\delta = \eps$. We will use this result in our proof of Theorem \ref{thm:main}, which we now provide.

\begin{proof}[Proof of Theorem \ref{thm:main}]
	Note that the optimal primal solution is unique, so we will denote this optimal solution as $x^\star$. Then according to Proposition \ref{lm upperbound}, it holds for $\thetax$ that
\begin{equation}\label{def thetax}
	\thetax \le  \inf_{x_{int} \in (\calF_p)_{++}}  \frac{\|x^\star - x_{int}\|}{\min_i |(x_{int})_i|} \ .
\end{equation}
For any $\delta > 0$, let $\bar{x}$ be 
\begin{equation}\label{def xbar}
	\begin{aligned}
		 \bar{x} \in \arg& \max_{x}  \quad \min_{i}\{x_i\}  \\ 
		& \ \ \operatorname{s.t.} \quad x \in \calF_p \\
		&  \quad \quad \quad \ \ c^\top x \le f^\star + \delta  \\
	\end{aligned}
\end{equation}
and then because of \eqref{def thetax} it holds that
\begin{equation}\label{def G}
	\thetax \le  \frac{\|x^\star - \bar{x}\|}{r_\delta} \ .
\end{equation}

Now, due to Lemma \ref{lm Rr lemma}, 
\begin{equation}\label{bound r delta}
	\frac{1}{r_\delta} \le \frac{R_\delta}{\delta} \ .
\end{equation}
Note that for any $p\in\mathbb{R}^n$, there exists the relation between norm $\|\cdot\|_1$ and the Euclidean norm $\|\cdot\|$ that $\|p\|_1 \le \sqrt{n}\|p\|$, so
$$
	\begin{aligned}
		R_\delta \ = \ & \max_{s}  \quad \| s \|_1  \quad \quad \quad \quad \quad \quad \le\ \  \max_{s} \ \ \sqrt{n}\cdot \|s\|\\ 
		& \ \ \operatorname{s.t.} \quad s \in \calF_d  \quad \quad \quad \quad   \quad \quad  \quad \operatorname{s.t.} \quad s \in \calF_d \\
		&  \quad \quad \quad \ \ -q^\top s \ge f^\star - \eps  \quad \quad \quad \quad \quad \ \ -q^\top s \ge f^\star - \eps  \\
	\end{aligned}
$$
which further implies
$$
\begin{aligned}
R_\delta & \le \sqrt{n} \cdot \left( \max_{s\in\calF_d, \ -q^\top s \ge f^\star - \delta} ~ \dist(s,\calS^\star) + \dist(0,\calS^\star) + \operatorname{Diam}(\calS^\star) \right) , \\
	& \ \le \sqrt{n} \cdot \left( \frac{\delta}{\|q\| \cdot \mu_d} + \dist(0,\calS^\star) + \operatorname{Diam}(\calS^\star) \right) \ .
\end{aligned} 
$$
Here the second inequality is due to the definition of the \LPsharp~$\mu_d$. And furthermore, $\|q\| = \dist(0,V_p)$ since \eqref{eq:norm_of_c_and_q}.
Now substitute the above inequality back to \eqref{bound r delta} and then  \eqref{def G} becomes
\begin{equation}\label{bound thetax}
		\thetax \le \|x^\star - \bar{x}\| \cdot \frac{\sqrt{n}}{\delta}\cdot \left(\frac{\delta}{\dist(0,V_p) \cdot \mu_d} + \dist(0,\calS^\star) + \operatorname{Diam}(\calS^\star)\right) \ .
\end{equation}

Similarly, since $c^\top \bar{x} - c^\top x^\star \le \delta$, according to the definition of the \LPsharp~$\mu_p$, it holds that:
\begin{equation}\label{bound dist to optimal}
		\|x^\star - \bar{x}\|  = \dist(\bar{x},\calX^\star)  \le \frac{\delta}{\|P_{\lin{V_p}}(c)\| \cdot \mu_p } =  \frac{\delta}{\dist(0,V_d) \cdot \mu_p } \ .
\end{equation}
Substituting \eqref{bound dist to optimal} back to \eqref{bound thetax} yields:
\begin{equation}
	\thetax \le \frac{1}{\dist(0,V_d) \cdot \mu_p } \cdot \sqrt{n} \cdot \left(\frac{\delta}{\dist(0,V_p) \cdot \mu_d} + \dist(0,\calS^\star) + \operatorname{Diam}(\calS^\star)\right) \ .
\end{equation}
Note that the $\delta$ in the above inequality can be any positive scalar, so
$$
\thetax \le \frac{\sqrt{n}}{ \mu_p }  \cdot \left(
\frac{\dist(0,\calS^\star)}{\dist(0,V_d)}
+
\frac{\operatorname{Diam}(\calS^\star)}{\dist(0,V_d)}
\right) \ .
$$
This is exactly \eqref{eq bound of thetax}. Now we have completed the proof.
\end{proof}

\subsection{The reciprocal of \LPsharp~cannot be upper-bounded by \oneratio}

In Theorem \ref{thm:main} we established an upper bound on the \oneratio~using the reciprocal of \LPsharp. In this subsection we provide an example where the reciprocal of \LPsharp~is huge but the \oneratio~and the distance to optimal solutions are both small. Consider the family of LP problems parameterized by $\gamma \in  [0,\pi/2)$ as follows:
$$
\begin{aligned}
	\mathrm{LP}_{\gamma}: \quad  & \min_{x \in \mathbb{R}^3} \ c_{\gamma}^\top x \quad 
	\\
	& \ \  \text{s.t.} \  \frac{\sqrt{3}}{3}\cdot \sum_{i=1}^3 x_i = 1  \ , \ x\ge 0 \ ,
\end{aligned}
$$ 
where $c_\gamma$ is defined as follows:
\begin{equation}
\begin{aligned}
	 c_{\gamma} :=   \cos(\gamma) \cdot \left(
	 \begin{array}{c}
	 	\frac{-1}{\sqrt{6}} \\
	 	 \frac{-1}{\sqrt{6}} \\
	 	  \frac{2}{\sqrt{6}}
	 \end{array}
	 \right)
	  + \sin(\gamma)  \cdot \left(
	  \begin{array}{c}
	  	 \frac{-1}{\sqrt{2}}\\
	  	 \frac{1}{\sqrt{2}} \\
	  	  0
	  \end{array}
	  \right) \ .
\end{aligned}
\end{equation}
For all $\gamma \in (0,\pi/2)$ the unique primal optimal solution is $x^\star = (\sqrt{3},0,0)$. However, when $\gamma = 0$ all points on the line segment connecting $(\sqrt{3},0,0)$ and $(0,\sqrt{3},0)$ are optimal solutions. Basically, as $\gamma \searrow 0$ the primal sharpness $\mu_p \searrow 0$ as well. We compute the \oneratio~and \LPsharp~for $\gamma \searrow 0$, which are shown in Figure \ref{fig:counterexample}. We see from Figure \ref{fig:counterexample} that $\mu_p$ decreases linearly in $O(\gamma)$ while $\mu_d$, $\thetax$, $\thetas$, and the primal and dual relative distances to optima all remain constant. This example shows that the reciprocal of \LPsharp~cannot be upper-bounded as a function of \oneratio~and the relative distance to optima.


\begin{figure}[htbp]
	\centering
	\includegraphics[width=1\linewidth]{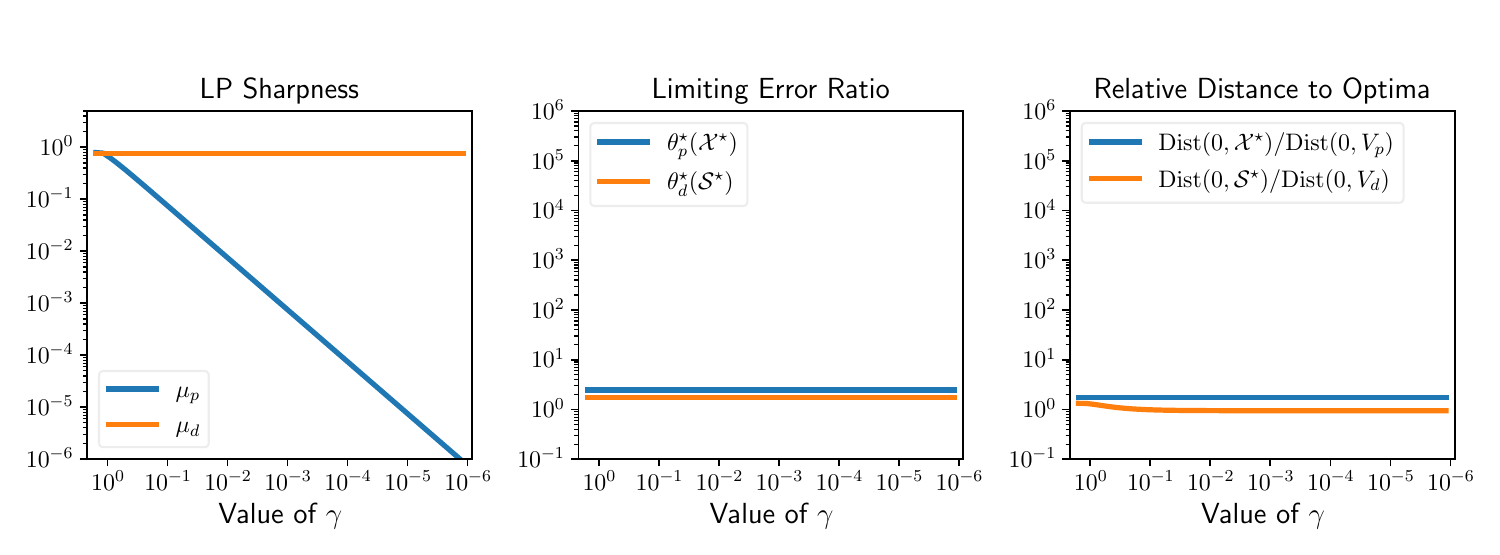}
	\caption{Condition measures of $\mathrm{LP}_\gamma$. }
	\label{fig:counterexample}
\end{figure}


\section{Complexity Implication in the Restarted PDHG}\label{secPDHGcomplexity}

In this section we show how Theorem \ref{thm:main} can be leveraged to yield a simplified complexity bound for restarted PDHG for LP. We first recall the complexity results in \cite{xiong2023geometry}. Let us assume that the affine subspace $V_p$ in \eqref{pro primal LP} is given by $\{x: Ax = b\}$ for a specific $A \in \mathbb{R}^{m\times n}$ and $b \in \mathbb{R}^m$. We define the following quantities:
\begin{equation}\label{eq  def lamdab min max}
	\lambda_{\max}:= \sigma_{\max}^+\left(A \right)\text{, }	\lambda_{\min}:= \sigma_{\min}^+\left(A \right), \text{ and }\kappa := \frac{\lambda_{\max}}{\lambda_{\min}}  \ ,
\end{equation}
and the measure of error of a primal-dual pair $(x,s)$ is the ``distance to optima'' error of $(x,s)$ given by:
$$
\calE_d(x,s) := \max \{ \dist(x,\calX^\star),\dist(s,\calS^\star)\} \ .
$$
The following is a re-statement of Theorem 3.1 of \cite{xiong2023geometry} whose result is for the ``standard'' step-sizes based on knowledge of $\lambda_{\min}$, $\lambda_{\max}$, $\|c\|$ and $\|q\|$ as follows.

\begin{lemma}[Theorem 3.1 of \cite{xiong2023geometry}]
	Suppose that Assumption \ref{assump:general} holds, and that the restarted PDHG (Algorithm 1 in \cite{xiong2023geometry}) is run starting from $z^{0,0} = (x^{0,0},y^{0,0} )= (0,0)$ using the $\beta$-restart condition with $\beta := 1/e$. Furthermore, let the step-sizes be chosen as follows:
	\begin{equation}\label{eq smart step size 1}
		\tau = \frac{\|q\|}{2\kappa\|c\|   } \ \ \text{ and } \ \ \sigma = \frac{\|c\|}{2\|q\|\lambda_{\max}\lambda_{\min}} \ .
	\end{equation}
	Let $T$ be the total number of PDHG iterations that are run in order to obtain a solution $(x,s)$ that satisfies $\calE_d(x,s) \le \eps$.  Then 
	\begin{equation}\label{eq overall complexity}
	T \ \le \ 5e \cdot  \condL \cdot \ln\Bigg(   
	8 e \cdot \condL \cdot \frac{\calE_d(x^{0,0},s^{0,0} )}{ \eps}\cdot  \Bigg(
	1+\kappa \frac{\|c\|}{\|q\|}
	\Bigg)
	\Bigg(
	1 + \frac{\|q\|}{\|c\|}
	\Bigg)  
	\Bigg) + 1\ ,
	\end{equation}
	where $\condL$ is defined as follows:
	\begin{equation}\label{eq easy L}
	\condL := 8.5 \kappa \left(
	\frac{1}{\mu_p} + \frac{1}{\mu_d} \right)
	\Bigg(  
	\thetax +\thetas   +  \frac{ \dist(0,\calX^\star) }{\dist(0, V_p)} +  \frac{\dist(c,\calS^\star)}{ \dist(0,V_d)} 
	\Bigg) \ .
\end{equation} 
\end{lemma}
In addition to the ``standard'' step-sizes described in \eqref{eq smart step size 1}, \cite{xiong2023geometry} also develops potentially better step-sizes that involve knowledge of \LPsharp, which we will call the ``optimized'' step-sizes. The complexity of using the ``optimized'' step-sizes is as follows.

\begin{lemma}[Theorem 3.2 in \cite{xiong2023geometry}]
The following choice of step-sizes:
\begin{equation}\label{eq smart step size 2}
	\tau = \frac{\mu_d\|q\|}{2\kappa\mu_p \|c\|   } \ \  \text{ and } \ \ \sigma = \frac{\mu_p\| c\|}{2\mu_d\|q\|\lambda_{\max}\lambda_{\min}} 
	\ 
\end{equation}leads to an alternative bound on the total number of PDHG iterations $T$:  
\begin{equation}\label{eq overall complexity2}
	T \ \le \ 5e \cdot  \condL \cdot \ln\Bigg(   
	8 e \cdot \condL \cdot \frac{\calE_d(x^{0,0},s^{0,0} )}{ \eps}\cdot  \Bigg(
	1+\kappa \frac{\mu_p\|c\|}{\mu_d\|q\|}
	\Bigg)
	\Bigg(
	1 + \frac{\mu_d\|q\|}{\mu_p\|c\|}
	\Bigg)  
	\Bigg) + 1\ ,
\end{equation}
using a structurally better value of the scalar $ \condL $, namely 
\begin{equation}\label{eq smart L}
	\condL  := 16 \kappa  \left(
	\frac{\thetax }{\mu_p}  +\frac{\thetas}{\mu_d } +
	\frac{ \dist(c,\calS^\star)}{\mu_p \cdot\dist(0,V_d)}    + 
	\frac{\dist(0,\calX^\star) }{\mu_d\cdot \dist(0, V_p)}
	\right)  \ .
\end{equation}
\end{lemma}

It can be observed when considering the terms outside the logarithm, that the number of iterations of restarted PDHG is $O(\calL)$, where $\calL$ is defined in \eqref{eq easy L} or \eqref{eq smart L} depending on the step-sizes used. However, $\calL$ in \eqref{eq easy L} and \eqref{eq smart L} involves \LPsharp, \oneratio, and the relative distance to optima. According to Theorem \ref{thm:main} and Corollary \ref{cor bound thetas}, we can conclude that if both the primal and dual optimal solutions are unique, then:
\begin{equation}
	\thetax \le \frac{\sqrt{n}}{ \mu_p }  \cdot 
	\frac{\dist(0,\calS^\star)}{\dist(0,V_d)} \quad \text{ and }\quad \thetas \le \frac{\sqrt{n}}{ \mu_d }  \cdot 
	\frac{\dist(0,\calX^\star)}{\dist(0,V_p)} \ .
\end{equation}
This allows us to provide an alternative expression for $\calL$ as follows.

\begin{corollary}
	When Assumption \ref{assump:general} holds, and both the primal and dual problems in \eqref{pro: primal dual reformulated LP} have unique optimal solutions, then the expression for $\calL$ in \eqref{eq easy L} can be replaced by:
	$$
	\condL := 8.5 \kappa (\sqrt{n}+ 1) \left(
	\frac{1}{\mu_p} + \frac{1}{\mu_d} \right)
	\Bigg(  
	 \frac{1}{ \mu_p }  \cdot \frac{\dist(0,\calS^\star)}{\dist(0,V_d)} +\frac{1}{ \mu_d }  \cdot 
	\frac{\dist(0,\calX^\star)}{\dist(0,V_p)} 
	\Bigg) \ ,
	$$
	and the expression for $\calL$ in \eqref{eq smart L}  can be replaced by
	$$
		\condL  := 16 \kappa (\sqrt{n}+1) \left(
		\frac{1}{ \mu_p^2 }  \cdot 
		\frac{\dist(0,\calS^\star)}{\dist(0,V_d)}  +\frac{1}{ \mu_d^2 }  \cdot 
		\frac{\dist(0,\calX^\star)}{\dist(0,V_p)}
		\right)  \ .
	$$
\end{corollary}

This allows for a worst-case complexity analysis of restarted PDHG without the direct use of \oneratio.

\bibliographystyle{plain}
\bibliography{reference}

\end{document}